\documentclass[a4paper,12pt,reqno]{amsart}
\usepackage{a4wide}
\usepackage{amsmath}
\usepackage{amssymb}
\usepackage{amsthm}
\usepackage{latexsym}
\usepackage{graphicx}
\usepackage[english]{babel}
\usepackage{algorithm}
\usepackage{algpseudocode}
\usepackage{algorithmicx}
         
\newtheorem{dfn} [subsection]{Definition}

\newtheorem{exm} [subsection]{Example}

\newtheorem{prop}[subsection]{Proposition}

\newtheorem{teor}[subsection]{Theorem}
\newtheorem{lema}[subsection]{Lemma}
\newtheorem{cor} [subsection]{Corollary}

\def\sdepth{\operatorname{sdepth}}
\def\qdepth{\operatorname{hdepth}}
\def\hdepth{\operatorname{hdepth}}

\def\depth{\operatorname{depth}}

\def\PP{\operatorname{P}}
\def\QQ{\operatorname{Q}}

\def\eq{\operatorname{eq}}

\begin{document}
\selectlanguage{english}
\frenchspacing

\numberwithin{equation}{section}

\title{On the arithmetic Hilbert depth}
\author{Silviu B\u al\u anescu$^1$ and Mircea Cimpoea\c s$^2$}
\date{}



\footnotetext[1]{ \emph{Silviu B\u al\u anescu}, University Politehnica of Bucharest, Faculty of
Applied Sciences, 
Bucharest, 060042, E-mail: silviu.balanescu@stud.fsa.upb.ro}
\footnotetext[2]{ \emph{Mircea Cimpoea\c s}, University Politehnica of Bucharest, Faculty of
Applied Sciences, 
Bucharest, 060042, Romania and Simion Stoilow Institute of Mathematics, Research unit 5, P.O.Box 1-764,
Bucharest 014700, Romania, E-mail: mircea.cimpoeas@upb.ro,\;mircea.cimpoeas@imar.ro}

\begin{abstract}
Let $h:\mathbb Z \to \mathbb Z_{\geq 0}$ be a nonzero function with $h(k)=0$ for $k\ll 0$. We define the
\emph{Hilbert depth} of $h$ by
$\qdepth(h)=\max\{d\;:\; \sum_{j\leq k} (-1)^{k-j}\binom{d-j}{k-j}h(j)\geq 0\text{ for all }k\leq d\}$.
We show that $\qdepth(h)$ is a natural generalization for the Hilbert depth of a subposet 
$\PP\subset 2^{[n]}$ and we prove some basic properties of it.

Given $h(j)=\begin{cases} aj^n+b,& j\geq 0 \\ 0, & j<0 \end{cases}$, with $a,b,n$ positive integers, we compute $\qdepth(h)$ for $n=1,2$
and we give upper bounds for $\qdepth(h)$ for $n\geq 3$. More generally, if $h(j)=\begin{cases} P(j),& j\geq 0 \\ 0,& j<0 \end{cases}$,
where $P(j)$ is a polynomial of degree $n$, with non-negative integer coefficients, and $P(0)>0$, we show that $\qdepth(h)\leq 2^{n+1}$.
\end{abstract}

\maketitle

\section*{Introduction}

Let $n$ be a positive integer. We denote $[n]=\{1,2,\ldots,n\}$. For $C\subset D\subset [n]$, the interval $[C,D]$ is the family 
$\{A\subset[n]\;:\;C\subset A\subset B\}$. We consider $2^{[n]}$ the poset of subsets of $[n]$ with the inclusion. 
Given a subposet $\PP$ of $2^{[n]}$, a \emph{Stanley decomposition} of $\PP$ is a decomposition 
$$\mathcal P:\;\PP=\bigcup_{i=1}^r [C_i,D_i],$$ into
disjoint intervals. The \emph{Stanley depth} of $\mathcal P$ is 
$$\sdepth(\mathcal P)=\max_{i=1}^r |D_i|$$ and \emph{the Stanley depth} of $\PP$ is
$$\sdepth(\PP)=\max\{\sdepth(\mathcal P)\;:\;\mathcal P\text{ a Stanley decomposition of }\PP\}.$$
A generalization of Stanley depth for arbitrary finite posets can be found in \cite{wang}.

The interest for this invariant came from commutative algebra, where a similar invariant was
associated to a multigraded $S$-module $M$, where $S=K[x_1,\ldots,x_n]$. 
To be more precise, a \emph{Stanley decomposition} of $M$ is a direct sum 
$$\mathcal D: M = \bigoplus_{i=1}^r m_i K[Z_i],$$ 
as a $\mathbb Z^n$-graded $K$-vector space, where $m_i\in M$ is homogeneous with respect to $\mathbb Z^n$-grading, 
$Z_i\subset\{x_1,\ldots,x_n\}$ such that $m_i K[Z_i] = \{um_i:\; u\in K[Z_i] \}\subset M$ is a free $K[Z_i]$-submodule of $M$. 
We define $\sdepth(\mathcal D)=\min_{i=1,\ldots,r} |Z_i|$ and 
$$\sdepth(M)=\max\{\sdepth(\mathcal D)|\;\mathcal D\text{ is a Stanley decomposition of }M\}.$$ 
Stanley \cite{stan} conjectured that $$\sdepth(M)\geq \depth(M),$$ 
a conjecture which was shown to be false in general, see Duval et.al \cite{duval}.

The connection between the combinatorial Stanley depth and the algebraic depth was made by Herzog Vl\u adoiu and Zheng in \cite{hvz}. They proved that 
if $I\subset J\subset S$ are two squarefree monomial ideals, then 
$$\sdepth(J/I)=\sdepth(\PP_{J/I}),\text{ where }\PP_{J/I}=\{A\subset [n]\;:\;\prod\limits_{j\in A}x_j\in J\setminus I\}.$$
Another related invariant, which came from the theory of Hilbert series, is the so called Hilbert depth; see \cite{uli}. 
Given $M$ a finitely generated graded $S$-module, the Hilbert depth of $M$, denoted by $\hdepth(M)$ is the maximal depth of a finitely generated graded $S$-module $N$
with the same Hilbert series as $M$. It is not difficult to see that, given a multigraded $S$-module $M$, then 
$$\hdepth(M)\geq \sdepth(M).$$
In \cite{lucrare2} we presented a new method of computing the Hilbert depth of a quotient of $J/I$ of two squarefree monomial ideals, as follows:

Let $\PP\subset 2^{[n]}$ be a subposet. For any $0\leq j\leq n$, we consider $\alpha_j(\PP)$ to be the number of subsets in $\PP$ with $j$ elements.
For any $0\leq k\leq d$, we define the numbers 
$$\beta_k^d(\PP):=\sum_{j=0}^k (-1)^{k-j} \binom{d-j}{k-j} \alpha_j(\PP).$$
If $\PP$ has a Stanley decomposition $\mathcal P$ with $\sdepth(\mathcal P)=d$, then, we note that $\beta_k^d(\PP)\geq 0$ for all $0\leq k\leq d$.
We defined the \emph{Hilbert depth} of $\PP$ by setting
$$\qdepth(\PP)=\max\{d\;:\;\beta_k^d(\PP)\geq 0\text{ for all }0\leq k\leq d\}.$$
From the aforementioned remark, it follows that $\sdepth(\PP)\leq \qdepth(\PP)$. In particular, if $\PP=\PP_{J/I}$ then, according to \cite[Theorem 2.4]{lucrare2},
we have that $$\hdepth(J/I)=\hdepth(\PP_{J/I}).$$
The aim of our paper is to further generalize the concept of Hilbert depth of a finite subposet of $2^{[n]}$ and to put more light on its combinatorial nature.

We denote $\mathbb Z_{\geq 0}$ the set of non-negative integers and 
$$\mathcal H=\{h:\mathbb Z\to \mathbb Z_{\geq 0}\text{ with }h(k)=0\text{ for }k\ll 0\}.$$
Given $h\in \mathcal H$ we define the \emph{Hilbert depth} of $h$ by
$$\qdepth(h)=\max\{d\;:\; \sum_{j\leq k} (-1)^{k-j}\binom{d-j}{k-j}h(j)\geq 0\text{ for all }k\leq d\}.$$
If $\PP\subset 2^{[n]}$ and $h_{\PP}(j)=\begin{cases} |\PP_j|,&0\leq j\leq n \\ 0,&\text{ otherwise}\end{cases}$ then it is clear that 
$$\qdepth(h_{\PP})=\qdepth(\PP).$$
Now, if $h:\mathbb Z \to \mathbb Z_{\geq 0}$ is a function as above, $k_0=\min\{j\;:\;h(j)>0\}$, $k_1=k_0+1$,
$k_f=\min\{j>k_0\;:\;h(j)=0\}-1$, $h_0=h(k_0)$, $h_1=h(k_1)$ and
$c(h)=\left\lfloor \frac{h_1}{h_0} \right\rfloor$ then, in Proposition \ref{p1} we prove that 
$$k_0\leq \qdepth(h) \leq \min\{k_f,k_0+c(h)\}.$$
In Proposition \ref{p17} we show that if $\qdepth(h)\geq d$ then 
$$h(k)\geq \binom{d-k_0}{k-k_0}h_0\text{ for all }k_0\leq k\leq d.$$
In Proposition \ref{p18} we prove that if 
$$h(k) \geq (d-k+1)h(k-1)\text{ for all }k_0+1\leq k\leq d,$$ 
then $\qdepth(h)\geq d$. 

As an application of the above properties, we show that for a geometric sequence, i.e. $h(j)=\begin{cases} a\cdot r^j, & j\geq 0 \\ 0,& j<0 \end{cases}$, where $a,r$ are positive integers, we have $\qdepth(h)=r$; see Corollary \ref{c19}.

Given $g,h\in \mathcal H$ and $c$ a positive integer, we show that
$$\qdepth(g+h)\geq \min\{\qdepth(g),\qdepth(h)\}\text{ and }\qdepth(ch)=\qdepth(h),$$
see Proposition \ref{p110} and Proposition \ref{p111}.

For $h\in\mathcal H$ and $m$ an integer, we define 
$$h[m](j):=h(m+j)\text{ for all }j\in \mathbb Z,$$ 
the $m$-shift of $h$. Given $h\in \mathcal H$ with $k_f(h)<\infty$, we prove that there exists an integer $m$ and a subposet $\PP\subset 2^{[n]}$
such that $h[m]=h_{\PP}$;
see Theorem \ref{teo1} and Theorem \ref{teo2}. Moreover, we show that we can choose $\PP$ such that $\qdepth(\PP)=\sdepth(\PP)$. 

In Section $3$ we study the Hilbert depth of certain polynomial sequences. As main tools, in order to obtain upper bounds for the Hilbert depth, 
we use Proposition \ref{p1} and, more important, we study the sign of the quadratic function $x\mapsto \beta_2^x(h)$.

In Theorem \ref{teop}, we show that if 
$h(j)=\begin{cases} P(j),& j\geq 0 \\ 0,& j<0\end{cases}$, 
where $P(j)$ is a polynomial of degree $n$, with non-negative integer coefficients, and $P(0)>0$, then 
$$\qdepth(h)\leq 2^{n+1}.$$
We believe that this bound is sharp, but can prove it only for $n=1$ and $n=2$.

We further study the function $h(j)=\begin{cases} aj^n+b, & j\geq 0 \\ 0,& j<0 \end{cases}$, where $a,b,n$ are positive integers.
We note that $\qdepth(h)\leq \left\lfloor \alpha \right\rfloor +1$, where $\alpha=\frac{a}{b}$.

In Proposition \ref{l31} we prove that
$$\qdepth(h)=c(h)=\lfloor \alpha \rfloor + 1\text{ for }\alpha<3\text{ and }\qdepth(h)\geq 3\text{ for }\alpha\geq 3.$$
In Theorem \ref{arit} (case $n=1$) we prove that if $h(j)= \begin{cases} aj+b, & j\geq 0 \\ 0, & j<0 \end{cases}$ then
$$ \qdepth(h) = \begin{cases} \lfloor \alpha \rfloor + 1,& a<3b \\ 4,& 3b\leq a\leq 4b \\ 3,& a>4b \end{cases}.$$
In Theorem \ref{pe2} (case $n=2$) we prove that if $h(j)= \begin{cases} aj^2+b, & j\geq 0 \\ 0, & j<0 \end{cases}$ then
$$\qdepth(h)=\begin{cases} \lfloor \alpha \rfloor + 1,& \alpha \in (0,7) \\
8, & \alpha \in [7,\frac{22}{3}] \\ 
7, & \alpha \in (\frac{22}{3},8] \\ 6, & \alpha \in (8,11] \\ 5, & \alpha \in (11,\infty) \end{cases},\text{ where }\alpha=\frac{a}{b}.$$
Also, we give an upper bounds for $\qdepth(h)$ for $n\geq 3$; see Theorem \ref{teo}. 
However, this limits is not sharp; see Example \ref{sarpe}.

Finally, in Section $4$ we present an algorithm to compute the Hilbert depth for a sequence $h(j)$ which takes positive values
for all $j$ in $\{0,1,\ldots,s\}$ and zero in rest.

\section{Basic properties}

We denote 
$$\mathcal H=\{h:\mathbb Z \to \mathbb Z_{\geq 0}\;:\;h\text{ is nonzero and }h(j)=0\text{ for }j\ll 0\}.$$
Let $k\leq d$ be two integers. We let
\begin{equation}\label{ec1}
\beta_k^d(h):=\sum_{j\leq k} (-1)^{k-j}\binom{d-j}{k-j}h(j)
\end{equation}
From \eqref{ec1} it follows that
\begin{equation}\label{ec2}
h(k)=\sum_{j\leq k}\binom{d-j}{k-j}\beta_j^d(h).
\end{equation}
With the above notations, we define:

\begin{dfn}
Let $h\in\mathcal H$. The \emph{(arithmetic) Hilbert depth} of $h$ is
$$\qdepth(h):=\max\{d\;:\;\beta_k^d(h)\geq 0\text{ for all }k\leq d\}.$$
\end{dfn}

\begin{lema}\label{recur}
For any $1\leq k\leq d$ we have that $\beta_{k}^{d+1}(h)=\beta_k^d(h)-\beta_{k-1}^d(h)$.
\end{lema}

\begin{proof}
From \eqref{ec1} it follows that
\begin{align*}
& \beta_{k}^{d+1}(h) - \beta_k^d(h) = \sum_{j=0}^k (-1)^{k-j} \left( \binom{d+1-j}{k-j} - \binom{d-j}{k-j} \right) h(j) = \\
& = \sum_{j=0}^k (-1)^{k-j} \binom{d-j}{k-1-j} h(j)= - \sum_{j=0}^{k-1} (-1)^{(k-1)-j}\binom{d-j}{(k-1)-j}h(j)=-\beta_{k-1}^d(h),
\end{align*}
hence we get the required result.
\end{proof}

\begin{prop}\label{p13}
If $d=\qdepth(h)$ then
$$\beta_k^{d'}(h)\geq \beta_k^d(h)\geq 0\text{ for all }0\leq k\leq d' \leq d.$$
\end{prop}

\begin{proof}
From Lemma \ref{recur}, it is enough to show that $\beta_k^{d'}(h)\geq 0$ for all $k\leq d'\leq d$.
Suppose this is not the case. We choose the greatest integer $d'$ with $d'<d$ such that there exists $k'\leq d'$ 
with $\beta_k^{d'}(h)<0$. We can assume also
that $k$ is the smallest integer with this property, that is $\beta_{\ell}^{d'}(h)\geq 0$ for all $\ell<k$. 
From Lemma \ref{recur} it follows that
$$\beta^{d'+1}_k(h) = \beta_k^{d'}(h) - \beta_{k-1}^{d'}(h)<0,$$
a contradiction.
\end{proof}

\begin{cor}\label{c14}
If $d=\qdepth(h)$ and $q>d$ then there exists $k\leq q$ such that $\beta^q_k(h)<0$.
\end{cor}

Let $h\in\mathcal H$. We denote
$$k_0(h)=\min\{j\;:\;h(j)>0\}\text{ and }k_{f}(h)=\min\{j\;:\;j\geq k_0(h)\text{ and }h(j+1)=0\}.$$
Note that $k_0(h)\in\mathbb Z$ and $k_0(h)\leq k_{f}(h)\leq +\infty$.
We also denote $$k_1(h):=k_0(h)+1\text{ and }c(h):=\left\lfloor \frac{h(k_1(h))}{h(k_0(h))} \right\rfloor.$$
With the above notations, we have that:

\begin{prop}\label{p1}
If $k_0:=k_0(h)$ and $k_f:=k_f(h)$ then:
$$k_0 \leq \qdepth(h)\leq \min\{ k_{f} , k_0 +c(h) \}.$$
\end{prop}

\begin{proof}
We denote $h_0:=h(k_0)$ and $h_1:=h(k_1)$. 
For any integers $k\leq d\leq k_0$, we have that $$\beta^d_k(h)=\begin{cases}h_0 ,&k=d=k_0 \\ 0,&\text{otherwise} \end{cases}.$$
Hence $k_0 \leq \qdepth(h)$. 

In order to prove the other inequality, let $d>k_0+c(h)+1$. Then
$$\beta^d_{k_1}(h) = h_1-\binom{d-k_0}{1}h_0 = h_1 - (c(h)+1)h_0 < h_1 - \frac{h_1}{h_0}h_0 = 0.$$
On the other hand from \eqref{ec2} we have that
$$0=h(k_f+1)=\sum_{j\leq d}\beta^{k_f+1}_j(h).$$
Since $\beta^d_{k_0}(h)=h_0>0$, it follows that there exists 
$k_0<k\leq d$ such that $\beta^d_k(h)<0$.
\end{proof}

\begin{cor}
With the above notations, $\qdepth(h)=k_0$ if and only if $h_0>h_1$.
\end{cor}

The following proposition gives necessary conditions for $\qdepth(h)\geq d$.

\begin{prop}\label{p17}
Let $h\in\mathcal H$ and $d=\qdepth(h)$. Let $k_0=k_0(h)$ and $h_0=h(k_0)$. 
Then, for any $k_0\leq k\leq d$, we have that $$h(k)\geq \binom{d-k_0}{k-k_0}h_0.$$
\end{prop}

\begin{proof}
It follows immediately from \eqref{ec2} and the fact that $\beta^d_{k_0}(h)=h_0$.
\end{proof}

The following result gives sufficient conditions for $\qdepth(h)\geq d$.

\begin{prop}\label{p18}
Let $h\in\mathcal H$, $k_0=k_0(h)$, $h_0=h(k_0)$ and $d > h_0$. If 
$$ h(k) \geq (d-k+1)h(k-1)\text{ for all }k_0+1\leq k\leq d,$$
then $\qdepth(h)\geq d$.
\end{prop}

\begin{proof}
Let $k$ with $k_0\leq k\leq d$. Then, from hypothesis, we have that
\begin{align*}
& \beta_k^d(h)=\sum_{j=k_0}^k (-1)^{k-j}\binom{d-j}{k-j}h(j) = \left( h(k)-\binom{d-k+1}{1}h(k-1) \right) + \\
& + \left( \binom{d-k+2}{2}h(k-2)-\binom{d-k+3}{3}h(k-3) \right) + \cdots \geq 0.
\end{align*}
Hence, $\qdepth(h)\geq d$, as required.
\end{proof}

\begin{cor}\label{c19}
Let $a$ and $r$ be two positive integers. Let $h(j)=\begin{cases} a\cdot r^j,& j\geq 0 \\ 0,& j<0 \end{cases}$ be the
geometric sequence of ratio $r$ and the first term $h(0)=a$. Then $\qdepth(h)=r$.
\end{cor}

\begin{proof}
Since $k_0(h)=0$ and $c(h)=\frac{h(1)}{h(0)}=r$, from Proposition \ref{p1} it follows that 
$$\qdepth(h)\leq r.$$
On the other hand, it is clear that
$$h(k) = rh(k-1) \geq (r-k+1)h(k-1)\text{ for all }1\leq k\leq r.$$
Hence, the conclusion follows from Proposition \ref{p18}.
\end{proof}

\begin{prop}\label{p110}
Let $h,g\in\mathcal H$. Then $$\qdepth(h+g)\geq \min\{\qdepth(h),\qdepth(g)\}.$$
\end{prop}

\begin{proof}
The conclusion follows from the fact that $\beta_k^d(h+g)=\beta_k^d(h)+\beta_k^d(g)$ for all $k\leq d$.
\end{proof}

\begin{prop}\label{p111}
Let $h\in\mathcal H$ and $c$ a positive integer. Then $$\qdepth(ch)=\qdepth(h).$$
\end{prop}

\begin{proof}
The conclusion follows from the fact that $\beta_k^d(cf)=c\beta_k^d(f)$ for all $k\leq d$.
\end{proof}

Let $h\in\mathcal H$ and $m\in\mathbb Z$. We define the $m$-shift function associated to $h$, by
$$h[m]:\mathbb Z\to\mathbb Z,\;h[m](k):=h(m+k)\text{ for all }k\in\mathbb Z.$$
We have the following:
The following result is straightforward. 

\begin{prop}\label{propo}
With the above notations we have that:
\begin{enumerate}
\item[(1)] $k_0(h[m])=k_0(h)-m$.
\item[(2)] $k_f(h[m])=k_f(h)-m$.
\item[(3)] $c(h[m])=c(h)$.
\item[(4)] $\qdepth(h[m])=\qdepth(h)-m$.
\end{enumerate}
\end{prop}

Let $k\in\mathbb Z$ be an integer. We denote 
\begin{align*}
& \mathcal H_{\geq k}=\{h\in\mathcal H\;:\;k_0(h)\geq k\}.\\
& \mathcal H^f=\{h\in\mathcal H\;:\;k_f(h) < \infty\}.\\
& \mathcal H_{\geq k}^f=\{h\in\mathcal H_{\geq k}\;:\;k_f(h) < \infty\}.
\end{align*}
We have the following:

\begin{cor}
Let $k\in\mathbb Z$ be an integer. If $h\in\mathcal H$ then there exists an integer $m$ such that $h[m]\in\mathcal H_{\geq k}$.
\end{cor}

\begin{proof}
Let $m\leq k_0(h)-k$ be an arbitrary integer. 
According to Proposition \ref{propo} we have $k_0(h[m])=k_0(h)-m\geq k$, thus $h[m]\in\mathcal H_{\geq k}$, as required.
\end{proof}

\begin{exm}\rm
Let $d\geq 1$ be an integer and 
$$h(j)=\begin{cases} \frac{d!}{(d-j)!},& 0\leq j\leq d \\ 0,& \text{ otherwise } \end{cases}.$$
Note that $k_0(h)=0$ and $k_f(h)=d$, therefore $h\in \mathcal H^f_{\geq 0}$. Also, $c(h)=\frac{h(1)}{h(0)}=d$.
From Proposition \ref{p1} it follows that $\qdepth(h)\leq d$. On the other hand, we have that
$$h(k) = (d-k)h(k-1) \text{ for all }1\leq k\leq d,$$
hence, from Proposition \ref{p18} it follows that $\qdepth(h)=d$.

Also, if we choose some arbitrary integer $m$, it is easy to see that
$$h[m](j)=h(m+j)=\begin{cases} \frac{d!}{(d-m-j)!},& -m \leq j\leq d-m \\ 0,& \text{ otherwise } \end{cases}.$$
Moreover, from Proposition \ref{propo} it follows that 
$$k_0(h[m])=-m,\; k_f(h[m])=d-m\text{ and }\qdepth(h[m])=d-m.$$
In particular, we have that $h[m]\in \mathcal H_{\geq -m}^f$.
\end{exm}

\section{Stanley depth and Hilbert depth of subposets of $2^{[n]}$}

First, we recall some notations and definitions from \cite{lucrare2}.

We denote $[n]:=\{1,2,\ldots,n\}$. 

For two subsets $C\subset D\subset [n]$, we denote $[C,D]:=\{A\subset [n]\;:\;C\subset A\subset D\}$,
      and we call it the \emph{interval} bounded by $C$ and $D$.

Let $\PP\subset 2^{[n]}$ be a nonempty family of subsets of $[n]$.

A partition of $P$ is a decomposition $\mathcal P:\;\PP=\bigcup_{i=1}^r [C_i,D_i]$,
      into disjoint intervals.

If $\mathcal P$ is a partition of $\PP$, we let $\sdepth(\mathcal P):=\min_{i=1}^r |D_i|$.

The \emph{Stanley depth} of $\PP$ is 
      $$\sdepth(\PP):=\max\{\sdepth(\mathcal P)\;:\;\mathcal P\text{ is a partition of }\PP\}.$$
For $0\leq k\leq n$, we let $\PP_k=\{A\in\PP\;:\; |A|=k\}$ and $\alpha_k(\PP):=|\PP_k|$.

For all $0\leq d\leq n$ and $0\leq k\leq d$, we consider the integers:
\begin{equation}\label{betak}
\beta_k^d(\PP):=\sum_{j=0}^k (-1)^{k-j} \binom{d-j}{k-j} \alpha_j(\PP).
\end{equation}
The \emph{Hilbert-depth} of $\PP$ is
      $$\qdepth(\PP):=\max\{d\;:\;\beta_k^d(\PP) \geq 0\text{ for all }0\leq k\leq d\}.$$
We associate to $\PP$ the function
$$h_{\PP}:\mathbb Z \to \mathbb \mathbb Z_{\geq 0},\; h_{\PP}(j):=\begin{cases} \alpha_j(\PP),& 0\leq j\leq n \\ 0,&\text{otherwise} \end{cases}.$$
It is clear that $h_{\PP}\in \mathcal H_0$. Moreover, we have the restrictions 
$$0\leq h_{\PP}(j)\leq \binom{n}{j}\text{ for all }j\in [n].$$
Obviously, the definition of the Hilbert depth of $\PP$ given above 
coincide with the definition of the $\qdepth(h_{\PP})$. Hence, we have the following:

\begin{prop}
With the above notations, we have that:
$$\sdepth(\PP) \leq \qdepth(\PP)=\qdepth(h_{\PP}).$$
\end{prop}

\begin{teor}\label{teo1}
Let $h\in\mathcal H$. Then:
\begin{enumerate}
\item[(1)] There exist two integers $m$ and $n\geq 1$ and a nonempty family $\PP\subset 2^{[n]}$ such that
$$d:=\qdepth(\PP)=\qdepth(h)-m\text{ and }h_{\PP}(j)=h[m](j)=h(m+j)\text{ for all }1\leq j\leq d.$$
\item[(2)] There exist $N\geq 1$ and $\PP'\subset 2^{[N]}$ such that $h_{\PP}=h_{\PP'}$ and
$$\sdepth(\PP')=\qdepth(\PP')=\qdepth(\PP).$$
\end{enumerate}
\end{teor}

\begin{proof}
(1) Let $m:=k_0(h)-1$. From Proposition \ref{propo}(1) it follows that $k_0(h[m])=1$, 
that is $h[m](j)=0$ for all $j\leq 0$. Let $d:=\qdepth(h[m])$. We choose $n\geq d+1$ with the property that
$$ h[m](j) = h(m+j) \leq \binom{n}{j} \text{ for all }1\leq j\leq d. $$
For instance, we can choose $n=\max\{h(m+1),h(m+2),\ldots,h(m+d)\}$, but lower values of $n$ could satisfy the above property.

Now, we can chose any subposet $\PP\subset 2^{[n]}$ with the property that $\emptyset\notin \PP$ and 
$$|\PP_k|=\begin{cases} h(m+k),& 1\leq k\leq d \\ 0, & k>d \end{cases}.$$ 
Since $$\beta_k^d(\PP)=\beta_k^d(h[m])\geq 0\text{ for all }k\leq d,$$
it follows that $\qdepth(\PP)\geq d$. On the other hand, since $\PP_{d+1}=\emptyset$ then $\qdepth(\PP)\leq d$.
Thus, the conclusion follows from Proposition \ref{propo}(4).

(2) For $1\leq j \leq d$, we let $b_j:=\beta^d_j(h_{\PP})$. We construct $\PP'$ as follows. Let
$$N:=b_1\cdot d + b_2\cdot (d-1) + \cdots + b_d \cdot 1\text{ and }r:=b_1+b_2+\cdots+b_d.$$
We let $\PP':=[C_1,D_1]\cup [C_2,D_2] \cup \cdots \cup [C_r,D_r]$, where the sets $C_j$ and $D_j$ are defined by:
\begin{align*}
& C_1=\{1\},\;D_1=\{1,2,\ldots,d\},\;C_2=\{d+1\},\;D_2=\{d+1,\ldots,2d\},\ldots,\\
& C_{b_1}=\{(d-1)b_1+1\},\;D_{b_1}=\{(d-1)b_1+1,\ldots,db_1\},\\
& C_{b_1+1}=\{db_1+1,db_1+2\},\;D_{b_1+1}=\{db_1+1,db_1+2,\ldots,db_1+(b_2-1)\},\\
& C_{b_1+2}=\{db_1+b_2,db_1+b_2+1\},\;D_{b_1+2}=\{db_1+b_2,\ldots,db_1+2(b_2-1)\},\ldots,\\
& C_{b_1+b_2}=\{db_1+(d-1)(b_2-1)+1, db_1+(d-1)(b_2-1)+2\},\\
& D_{b_1+b_2}=\{db_1+(d-1)(b_2-1)+1,\ldots,db_1+(d-1)b_2\},\ldots,\\
& \ldots, C_{r-j}=D_{r-j}=\{N-j\}\text{ for }0\leq j\leq b_d-1.
\end{align*}
In general, for $0\leq i\leq r-1$ and $1\leq j\leq b_{i+1}$ we have that
\begin{align*}
& C_{b_1+\cdots+b_i+j}=\{\ell(i,j),\ell(i,j)+1,\cdots,\ell(i,j)+i\}\text{ and }\\ 
& D_{b_1+\cdots+b_i+j}=\{\ell(i,j),\ell(i,j)+1,\cdots,\ell(i,j)+d-1\},\\
& \text{ where }\ell(i,j)=b_1\cdot d+b_2\cdot (d-1)+ \cdots + b_i\cdot (d-i) + (j-1)\cdot (d-i-1) + 1.
\end{align*}
From construction, $\PP'$ satisfy the required conditions.
\end{proof}

Note that, since $\qdepth(h)$ is bounded by $c(h)+k_0$, 
only a finite number of values of $h$ determines the invariant. 
Hence, according to the previous result, we can reduce the study of Hilbert depth of functions 
to the study of the Hilbert depth of a subposets of $2^{[n]}$.

\begin{teor}\label{teo2}
Let $h\in\mathcal H^f$. Then, there exist $m\in\mathbb Z$, $n\geq 1$ and $\PP\subset 2^{[n]}$ such that $h_{\PP}=h[m]$.
Moreover, we can choose $N\geq 1$ and $\PP'\subset 2^{[N]}$ such that $h_{\PP'}=h_{\PP}$ and 
$$\sdepth(\PP')=\qdepth(\PP')=\qdepth(\PP).$$
\end{teor}

\begin{proof}
The proof is similar to the proof of Theorem \ref{teo1} so we give just a sketch of it.
We choose $$n:=\max\{k_f(h)-k_0(h)+2,\; h(j)\;:\;j\in\mathbb Z\}\text{ and }m:=k_0(h)-1.$$
It follows that $h[m](j)=0\text{ for }j\leq 0\text{ and }j\geq n.$ 
Also, from our choice, we have $$h[m](j)\leq n \leq \binom{n}{j}\text{ for }1\leq j\leq n-1.$$
Hence, we can find a subposet $\PP\subset 2^{[n]}$ with $h_{\PP}=h[m]$, as required.

Similar to the proof of Theorem \ref{teo1}(2), we let $b_j=\beta^d_j(h_{\PP})$ for $1\leq j\leq d=\qdepth(h[m])$. 
We let $$N:=\max\{n,b_1\cdot d + b_2\cdot (d-1) + \cdots + b_d \cdot 1,\;h[m](j)\;:\;d+1\leq j\leq n-1\}.$$
We define $\PP'\subset 2^{[N]}$ by $\PP':=\bigcup_{i=1}^r [C_i,D_i] \cup \QQ$, where $C_i$ and $D_i$ are defined as in the
proof of Theorem \ref{teo1}(2) and $\QQ\subset 2^{[N]}$ such that $h_{\QQ}(j)=\begin{cases} h[m](j),&d+1\leq j\leq N-1 \\ 0,&\text{otherwise} \\ \end{cases}$.
From definition, we have that $h_{\PP'}=h_{\PP}=h[m]$, hence $\qdepth(\PP')=\qdepth(\PP)$.
We consider the decomposition $$\mathcal P':\PP'=\bigcup_{i=1}^r [C_i,D_i] \cup  \bigcup_{D\in\QQ}[D,D].$$ Since $|D_i|=d$ for all $1\leq i\leq r$
and $|D|\geq d+1$ for all $D\in\QQ$, it follows that $\sdepth(\mathcal P')=d$ and thus $\sdepth(\PP')\geq d$. On the other hand, we have that
$$\sdepth(\PP')\leq \qdepth(\PP')=\qdepth(\PP)=d.$$
Thus, we get the required conclusion.
\end{proof}

\begin{exm}\rm
Let $h:\mathbb Z\to\mathbb Z_{\geq 0}$, $h(-2)=2$, $h(-1)=4$, $h(0)=7$, $h(1)=3$, $h(2)=1$ and $h(j)=0$ for $j\leq -3$ and $j\geq 3$.
Obviously, we have $k_0=k_0(h)=-2$, $k_f=k_f(h)=2$, $c=c(h)=\frac{4}{2}=2$ and $m=k_0-1=-3$. Let $g:=h[-3]$.

From Proposition \ref{p1} it follows that $\qdepth(h)\leq 0$ and thus $d:=\qdepth(g)\leq 3$. On the other hand, we have
$$\beta^3_1(g)=g(1)=2,\;\beta^3_2(g)=g(2)-\binom{2}{1}g(1)=0\text{ and }\beta^3_3(g)=g(3)-g(2)+g(1)=7-4+2=5,$$
thus $d=\qdepth(g)=3$. Let $n=\max\{k_f-k_0+2,\;h(j)\;-2\leq j\leq 2\}=\max\{6,7\}=7$. We can chose $\PP\subset 2^{[7]}$ as follows:
\begin{align*}
& \PP=\{ \{1\},\;\{2\},\;\{1,2\},\;\{1,3\},\;\{2,3\},\;\{1,4\},\;\{1,2,3\},\;\{1,2,4\},\;\{1,2,5\},\;\{1,3,4\},\;\{1,3,5\},\; \\
& , \{1,4,5\},\;\{2,4,5\},\;\{1,2,3,4\},\;\{1,2,3,5\},\;\{1,3,4,5\},\;\{1,2,3,4,5\}\}.
\end{align*}
It is easy to check that $g=h_{\PP}$. Moreover, since the intervals $[\{1\},\{1,3,4\}]$ and $[\{2\},\{2,3,4\}]$ are disjoint and cover
all the sets with $1$ and $2$ elements of $\PP$, it is easy to construct a Stanley decomposition 
$\mathcal P$ of $\PP$ with $\sdepth(\mathcal P)=3$.
Therefore, $\sdepth(\PP)=\qdepth(\PP)=3$. Of course, for a different choice of $\PP$, one could have $\sdepth(\PP)<\qdepth(\PP)$.
\end{exm}

\section{Hilbert depth of polynomial sequences}

First, we prove the following result:

\begin{teor}\label{teop}
Let $n\geq 1$ be an integer. 
For any polynomial $P\in\mathbb Z_{\geq 0}[X]$ of degree $n$ with $P(0)>0$,
if $h(j)=\begin{cases} P(j),& j\geq 0 \\ 0,& j<0 \end{cases}$, then $$\qdepth(h)\leq 2^{n+1}.$$
\end{teor}

\begin{proof}
Assume $P=a_nX^n+\cdots+a_1X+a_0$ with $a_i\in \mathbb Z$. Note that $a_n>0$ and $a_0=P(0)>0$. We
denote $\alpha_j=\frac{a_j}{a_0}$ for all $1\leq j\leq n$. From Proposition \ref{p1} we have that
\begin{equation}\label{cehas}
\qdepth(h)\leq c(h)=\left\lfloor \frac{P(1)}{P(0)} \right\rfloor =\left\lfloor \alpha_1+\cdots+\alpha_n \right\rfloor + 1 \leq \alpha_1+\cdots+\alpha_n+1 .
\end{equation}
Let $d\geq 2$ be an integer. Then
\begin{equation}\label{beta2d}
\beta_2^d(h)=\binom{d}{2}P(0)-\binom{d-1}{1}P(1)+\binom{d-2}{0}P(2)=\frac{P(0)}{2}d^2 - d\left(\frac{P(0)}{2}+P(1)\right)+P(1)+P(2).
\end{equation}
Let 
$$\Delta:=\left(\frac{P(0)}{2}+P(1)\right)^2-2P(0)(P(1)+P(2))=P(1)^2 + \frac{P(0)^2}{4} - P(0)P(1) - 2P(0)P(2).$$  
We have that
$$\Delta':=\frac{\Delta}{a_0^2} = (1+\alpha_1+\cdots+\alpha_n)^2 + \frac{1}{4} - (1+\alpha_1+\cdots+\alpha_n) - 2 (1+2\alpha_1+
4\alpha_2+\cdots+2^{n}\alpha_n) = $$
\begin{equation}\label{deltap}
= (\alpha_1+\cdots+\alpha_n)^2 - (2^2-1) \alpha_1 - (2^3-1) \alpha_2 - \cdots - (2^{n+1}-1)\alpha_n - \frac{7}{4}. 
\end{equation}
Note that, if $\Delta'>0$ (which is equivalent to $\Delta>0$) then, from \eqref{beta2d}, it follows that for 
any integer $d$ with
$$d \in (d_1,d_2)\text{ where }d_{1,2}=\frac{3}{2}+\alpha_1+\cdots+\alpha_n \pm \sqrt{\Delta'},$$ 
we have that $\beta_2^d(h)<0$ and therefore $\qdepth(h)<d$.

If $\alpha_1+\cdots+\alpha_n\geq 2^{n+1}$ then from \eqref{deltap} it follows that 
\begin{align*}
& \Delta' \geq (\alpha_1+\cdots+\alpha_n)^2 - (2^{n+1}-1)(\alpha_1+\cdots+\alpha_n) -\frac{7}{4} \geq \\
& \geq (\alpha_1+\cdots+\alpha_n)^2 - 2^{n+2}(\alpha_1+\cdots+\alpha_n)+2^{n+1}(\alpha_1+\cdots+\alpha_n)+(\alpha_1+\cdots+\alpha_n)-\frac{7}{4} > \\
& > (\alpha_1+\cdots+\alpha_n)^2 - 2^{n+2}(\alpha_1+\cdots+\alpha_n) + 2^{n+1}\cdot 2^{n+1} + 2^{n+1}-\frac{7}{4} = \\
& = (\alpha_1+\cdots+\alpha_n-2^{n+1})^2 + 4 - \frac{7}{4} \geq \frac{9}{4}.
\end{align*}
In particular, it follows that 
\begin{equation}\label{deunu}
d_1:= \frac{3}{2}+\alpha_1+\cdots+\alpha_n - \sqrt{\Delta'} \leq 2^{n+1}\text{ and }d_2-d_1 \geq 3.
\end{equation}
The conclusion follows from \eqref{cehas}, \eqref{beta2d}, \eqref{deunu} and the above considerations.
\end{proof}

Let $a,b,n$ be three positive integers. We consider the sequence 
$$h\in\mathcal H_0,\; h(j)=aj^n+b\text{ for all }j\geq 0.$$
We denote $\alpha:=\frac{a}{b}$. Note that, by Proposition \ref{p111}, $\qdepth(h)$ depends only on $\alpha$ and $n$.

We have 
$$c(h)=\left\lfloor \frac{h(1)}{h(0)} \right\rfloor = \left\lfloor \alpha \right \rfloor +1.$$
Since $k_0(h)=0$ and $k_f(h)=\infty$, from Proposition \ref{p1} it follows that 
\begin{equation}\label{ich}
0\leq \qdepth(h)\leq c(h)=\left \lfloor \alpha \right \rfloor +1\leq \alpha+1,
\end{equation}
where  With the above notations, we have:

\begin{lema}\label{bdd}
For any integer $d\geq 1$ we have that $\beta_d^d(h)> 0$.
\end{lema}

\begin{proof}
We have that
$$\beta_d^d(h)=h(d)-h(d-1)+h(d-2)-h(d-3)+\cdots = a(d^n-(d-1)^n+(d-2)^n-(d-3)^n+\cdots)> 0,$$
as required.
\end{proof}

\begin{prop}\label{l31}
\begin{enumerate}
\item[(1)] If $a<3b$ then $\qdepth(h)=c(h)$.
\item[(2)] If $a\geq 3b$ then $\qdepth(h)\geq 3$.
\end{enumerate}
\end{prop}

\begin{proof}
(1) If $a<b$, i.e. $c(h)=1$, then 
$$\beta_0^1(h)=h(0)=b>0\text{ and }\beta_1^1(h)=h(1)-h(0)=a,$$
hence $\qdepth(h)\geq 1$. Thus, we are done by \eqref{ich}.

If $b\leq a < 2b$, i.e. $c(h)=2$, then
$$ \beta_0^2(h)=h(0)=b>0,\;\beta_1^2(h)=h(1)-2h(0)=a-b\geq 0.$$
Also, from Lemma \ref{bdd} we have that $\beta_2^2(h)> 0$.
Hence $\qdepth(h)\geq 2$ and the conclusion follows from \eqref{ich}.

If $2b\leq a <3b$, i.e. $c(h)=3$, then
\begin{align*}
& \beta_0^3(h)=h(0)=b>0,\;\beta_1^3(h)=h(1)-3h(0)=a-2b\geq 0\text{ and }\\
& \beta_2^3(h)=3h(0)-2h(1)+h(2)=a(2^n-2)+2b\geq 2>0
\end{align*}
Also, from Lemma \ref{bdd} we have that $\beta_3^3(h)>0$.
Hence $\qdepth(h)\geq 3$ and the conclusion follows from \eqref{ich}.

(2) The proof is similar to the proof of the case $2b\leq a<3b$.
\end{proof}

Now, assume $c(h)\geq 4$ and let $4\leq d\leq c(h)$. We have that
$$\beta_2^d(h)=\binom{d}{2}h(0)-\binom{d-1}{1}h(1)+\binom{d-2}{0}h(2)=\frac{bd(d-1)}{2}-(d-1)(a+b)+2^n a+b.$$
By straightforward computations, we get
\begin{equation}\label{fd}
f(d):=\beta_2^d(h)=\frac{b}{2}d^2 - \left(a+\frac{3}{2}b\right)d + (2^n+1)a+2b.
\end{equation}
The discriminant of the quadratic equation $f(x)=0$ is
\begin{equation}\label{del}
\Delta = a^2-(2^{n+1}-1)ab-\frac{7}{4}b^2
\end{equation}
Note that the sign of $\Delta$ is the sign of the quadratic expression
$$h(\alpha)=\alpha^2 -(2^{n+1}-1)\alpha-\frac{7}{4}.$$ 
Since the discriminant of $h$ is 
$\Delta':=(2^{n+1}-1)^2+7=2^{2n+2}-2^{n+2}+8,$
it follows that 
\begin{equation}\label{alfa1}
\Delta\leq 0\text{ and thus }f(d)\geq 0\text{ for }\alpha \leq \alpha_1:= 2^n + \sqrt{2^{2n}-2^n+2} - \frac{1}{2}.
\end{equation}
For $\alpha>\alpha_1$, we let $g(\alpha)=$ the lowest solution of the equation $f(x)=0$, that is
$$g(\alpha):=\alpha+\frac{3}{2} - \sqrt{\alpha^2 - (2^{n+1}-1)\alpha-\frac{7}{4}}.$$
Now, we can prove the following result:

\begin{lema}\label{cheie}
Let $d\leq c(h)$ be an integer.
\begin{enumerate}
\item[(1)] If $\alpha\leq \alpha_1$ then $\beta_2^d(h)\geq 0$.
\item[(2)] If $\alpha>\alpha_1$ and $d \leq g(\alpha)$, then $\beta_2^d(h)\geq 0$.
\item[(3)] If $\alpha>\alpha_1$ and $d > g(\alpha)$, then $\qdepth(h)<d$.
\end{enumerate}
\end{lema}

\begin{proof}
(1) It is clear.

(2) From the above considerations, the hypothesis implies $\beta_2^d(h)=f(d)\geq 0$.

(3) It follows from the fact that 
$$g(\alpha)< d\leq c(h) < \alpha+\frac{3}{2} + \sqrt{\alpha^2 - (2^{n+1}-1)\alpha-\frac{7}{4}},$$
\eqref{fd} and \eqref{del}.
\end{proof}

Now, we consider the function $g(x)$, defined on $(\alpha_1,\infty)$. By straightforward computations,
we can show that $g'(x)<0$ for all $x\in (\alpha_1,\infty)$ and thus $g(x)$ is decreasing. 

On the other hand, we have 
\begin{equation}\label{limite}
\lim_{x\searrow \alpha_1}g(x)=\alpha_1+\frac{3}{2}\text{ and }\lim_{x\to\infty}g(x)=2^n+1.
\end{equation}
It follows that $g(x)$ takes values in the interval $(2^n+1,\alpha_1+\frac{3}{2})$. 

\subsection*{The case $n=1$}

If $n=1$ then, from \eqref{alfa1} and \eqref{limite}, it follows that
\begin{equation}\label{caz1}
\alpha_1=\frac{7}{2},\; \lim_{x\to\alpha_1}g(x)=5,\; g(4)=4\text{ and }\lim_{x\to\infty}g(x)=3.
\end{equation}
Now, we have all the ingredients in order to prove the following theorem:

\begin{teor}\label{arit}
Let $h\in\mathcal H_0$, $h(j)=aj+b$ for all $j\geq 0$, be an arithmetic sequence with $a,b$ positive integers. Then:
$$ \qdepth(h) = \begin{cases} 1,& a<b \\ 2,& b\leq a < 2b \\ 3,& 2b\leq a < 3b \\ 4,& 3b\leq a\leq 4b \\ 3,& a>4b \end{cases}.$$
\end{teor}

\begin{proof}
If $a<3b$ then the conclusion follows from Proposition \ref{l31}(1) and the fact that $c(h)=\left\lfloor \frac{a}{b} \right\rfloor + 1$.
Now, assume that $a\geq 3b$. We have that $c(h)\geq 4$ and $\alpha=\frac{a}{b}\geq 3$. Also, from Proposition \ref{l31}(2) we have
that $\qdepth(h)\geq 3$. If $a>4b$, i.e. $\alpha>4$, then from \eqref{caz1} and the fact that $g(x)$ is decreasing it follows that
 $g(\alpha)<4$. Hence, Lemma \ref{cheie} implies $\qdepth(h)<4$ and thus $\qdepth(h)=3$.

In order to complete the proof, assume that $3b\leq a\leq 4b$. Again, since $\alpha\in [3,4]$, from \eqref{caz1} we have that
$g(\alpha)<5$ and thus, from Lemma \ref{cheie} it follows that $\qdepth(h)\leq 4$. On the other hand, we have that 
$\beta^4_0(h)=h(0)=b>0$. Since $\alpha\geq 3$, we also have $\beta^4_1(h)\geq 0$. From Lemma \ref{cheie} it follows that $\beta^4_2(h)\geq 0$.
Also, Lemma \ref{bdd} implies $\beta^4_4(h)>0$ and
$$\beta^4_3(h)=h(3)-2h(2)+3h(1)-4h(0)=2(a-b)>0.$$
Hence, $\qdepth(h)=4$, as required.
\end{proof}

Now, we consider the case $n\geq 2$. First, we give a sharp estimation for $\alpha_1$.

\begin{lema}
 If $n\geq 2$ then $\alpha_1\in (2^{n+1}-1,2^{n+1}-\frac{1}{2})$ and thus 
 $$\lim_{x\searrow \alpha_1}g(x) \in (2^{n+1}+\frac{1}{2},2^{n+1}+1).$$
\end{lema}

\begin{proof}
It follows from the inequalities
$$ \left(2^n-\frac{1}{2}\right)^2 < 2^{2n}-2^n+2 < 2^{2n} \text{ for all }n\geq 2,$$
and \eqref{limite}.
\end{proof}

\begin{lema}\label{liema}
If $n\geq 2$ and $2 \leq m\leq 2^n$ then the equation $g(x)=2^{n}+m$ has the (unique) solution
$$\lambda_{2^n+1-m}:=\frac{m^2 + m(2^{n+1}-3) +2^{2n} -3\cdot 2^n +4}{2m-2}.$$
Moreover we have that $\lambda_0<\lambda_1<\cdots<\lambda_{2^n-1}$.
\end{lema}

\begin{proof}
It follows by straightforward computations. The uniqueness of the solution and the inequalities 
$\lambda_0<\lambda_1<\cdots<\lambda_{2^n-1}$ are consequences of the fact that $g(x)$ is decreasing.
\end{proof}

As above, $\alpha=\frac{a}{b}$ and $c(h)=\lfloor \alpha \rfloor + 1$. We consider the expression 
$$\eq(h):=\begin{cases} c(h), & \alpha \in (0,2^{n+1}-1) \\ 2^{n+1}, & \alpha \in [2^{n+1}-1,\alpha_1] \\
                        2^{n+1}-1, & \alpha \in (\alpha_1,\alpha_2], \\ \vdots \\ 2^{n}+2, & \alpha \in (\alpha_{2^n-2},\alpha_{2^n-1}] \\
												2^{n}+ 1, & \alpha \in (\alpha_{2^n-1},\infty) \end{cases}.$$
We use also the following lemma:

\begin{lema}\label{q4}
If $n\geq 2$ and $c(h)\geq 4$ then $\qdepth(h)\geq 4$.
\end{lema}

\begin{proof}
We apply Proposition \ref{l31}
and we use a similar argument as in the last part of proof of Theorem \ref{arit}.
\end{proof}

\begin{teor}\label{teo}
Let $n\geq 2$. With the above notations, we have that 
$$\qdepth(h)\leq \eq(h).$$ 
Moreover, the equality holds for $c(h)\leq 4$.
\end{teor}

\begin{proof}
It follows from \eqref{ich}, Lemma \ref{cheie}, Lemma \ref{liema} and Lemma \ref{q4}.
\end{proof}

\begin{teor}\label{pe2}
If $h(j)=\begin{cases} aj^2+b,& j\geq 0 \\ 0,& j<0 \end{cases}$, where $a,b$ are positive integers, then:
$$\qdepth(h)=\begin{cases} \lfloor \alpha \rfloor + 1,& \alpha \in (0,7) \\
8, & \alpha \in [7,\frac{22}{3}] \\ 
7, & \alpha \in (\frac{22}{3},8] \\ 6, & \alpha \in (8,11] \\ 5, & \alpha \in (11,\infty) \end{cases},\text{ where }\alpha=\frac{a}{b}.$$
\end{teor}

\begin{proof}
From Lemma \ref{liema} it follows that $\lambda_1=\frac{22}{3}$, $\lambda_2=8$ and $\lambda_3=11$. Hence, the inequality $\leq$ follows from 
Theorem \ref{teo}. In order to complete the proof, it is enough to show the following assertions:
\begin{enumerate}
\item[(i)] If $\alpha \geq 4$ then $\qdepth(h)\geq 5$.
\item[(ii)] If $\alpha \in [5,11]$ then $\qdepth(h)\geq 6$.
\item[(iii)] If $\alpha \in [6,8]$ then $\qdepth(h)\geq 7$.
\item[(iv)] If $\alpha \in [7,\frac{22}{3}]$ then $\qdepth(h)\geq 8$.
\end{enumerate}

(i) We have that $\beta_0^5(h)=b>0$. Also, since $c(h)\geq 5$ it follows that $\beta_1^5(h)\geq 0$. 
    Lemma \ref{cheie} implies $\beta_2^5(h)\geq 0$ and Lemma \ref{bdd} implies $\beta_5^5(h)\geq 0$.
		On the other hand, by straightforward computations we have:
		$$\beta_3^5(h)=3a-6b>0\text{ and }\beta_3^5(h)=6a+5b>0.$$
(ii) As in the case (i), it is enough to show the following:
     $$\beta_3^6(h)=3a-13b > 3(a-5b) \geq 0,\; \beta_4^6(h)=3a+9b>0,\;\beta_5^6(h)=9a-3b>0.$$
(iii) By straightforward computations, we get
      $$\beta_3^7(h)=4a-24b \geq 0,\;\beta_4^7(h)=22b>0,\;\beta_5^7(h)=5a-12b>0\text{ and }\beta_6^7(h)=12a+4b>0.$$
			The conclusion follows similar to the case (i).\\
(iv) By straightforward computations, we get
      \begin{align*}
			& \beta_3^8(h)=6a-40b \geq 0,\;\beta_4^8(h)=-4a+46b = 4b(\frac{23}{2}-\alpha)>0\text{ since }\alpha\leq \frac{22}{3},\\
			& \beta_5^8(h)=6a-26b >0,\;\beta_6^8(h)=6a+15b>0\text{ and }\beta_7^8(h)=16a-4b>0.
			\end{align*}
			The conclusion follows similar to the case (i).
\end{proof}

\begin{exm}\label{sarpe}\rm
It is natural to ask if the upper bound given in Theorem \ref{teo} holds for $n\geq 3$. Unfortunately, this is not the case 
even for $n=3$. Let $h(j)=\begin{cases} 15j^3 + 1, & j\geq 0 \\ 0, & j<0 \end{cases}$. Since $\alpha=15=2^{4}-1$, it follows that $\eq(h)=16$.
On the other hand, we have that $$\beta_3^16(j)=-\binom{16}{3}+\binom{15}{2}\cdot 16 - \binom{14}{1}\cdot 121 + \binom{13}{0}\cdot 405 = -168<0,$$
and, therefore, $\qdepth(h)<16$.
\end{exm}

\newpage
\section{An algorithm to compute the arithmetic Hilbert depth}

\begin{algorithm}
\caption{Compute the Hilbert depth of $h=(h(0),h(1),\ldots,h(s))$, $h(j)>0$ for $0\leq j\leq s$}\label{alg:cap}
\begin{algorithmic}
\State $d \gets \lfloor h(1)/h(0) \rfloor$
\If{$d>s$} \State $d=s-1$ \EndIf
\State $good \gets\textbf{false}$
\While{$(!good\text{ and }d>0)$}
\State $good \gets\textbf{true}$
\For{$k:=1\textbf{ to }d$}
  \State $\beta \gets 0$
	\For{$j:=1\textbf{ to }k$}
	   \State $\beta \gets \beta+(-1)^{k-j}\cdot Binomial(d-j,k-j)$
	\EndFor
	\If{$\beta<0$}
	   \State $good\gets\textbf{false}$ \\
		 \hspace{52 pt}\textbf{break}
	\EndIf	
\EndFor
\If{$!good$} 
\State $d=d-1$
\EndIf
\EndWhile \\
\Return $d$
\end{algorithmic}
\end{algorithm}

\subsection*{Aknowledgments} 

The second author, Mircea Cimpoea\c s, was supported by a grant of the Ministry of Research, Innovation and Digitization, CNCS - UEFISCDI, 
project number PN-III-P1-1.1-TE-2021-1633, within PNCDI III.





\end{document}